\documentclass[draft]{amsart}
\usepackage{amssymb,latexsym,amscd, amsthm, epsfig,amsmath,amssymb,cancel}
\usepackage{color}
\usepackage{graphics}
\usepackage{graphicx}
\usepackage{xypic}
\usepackage[all]{xy}
\usepackage{verbatim}
\usepackage{hyperref}
\usepackage{mathrsfs}

 \newtheorem{teo}{Theorem}[section]
 \newtheorem{corollary}[teo]{Corollary}
 \newtheorem{lem}[teo]{Lemma}
 \newtheorem{prop}[teo]{Proposition}
 \newtheorem{conj}{Conjecture}
 \theoremstyle{definition}
 \newtheorem{dfnt}[teo]{Definition}
 \newtheorem{remark}[teo]{Remark}
 \theoremstyle{definition}
 \newtheorem{example}[teo]{Example}

 \newcommand{\XX}{\mathbb{X}}
 \newcommand{\PP}{\mathbb{P}}

\def\move-in{\parshape=1.75true in 5true in}

\newcommand{\G}{\mathcal{G}}
\newcommand{\CC}{\mathbb{C}}
\newcommand{\ZZ}{\mathbb{Z}}
\newcommand{\NN}{\mathbb{N}}
\newcommand{\bg}{\mathbf{g}}
\newcommand{\bh}{\mathbf{h}}
\newcommand{\bx}{\mathbf{x}}
\newcommand{\HF}{\mathrm{HF}}
\newcommand{\HS}{\mathrm{HS}}
\newcommand{\mcdeg}{\mathrm{mcdeg}}
\newcommand{\partition}{\mathrm{part}}
\newcommand{\wt}{\mathrm{wt}}
\newcommand{\Q}{\mathcal{Q}}
\newcommand{\cG}{\mathcal{G}}

\title{On a class of power ideals}

\author[J. Backelin]{J\"orgen Backelin}
\address[J. Backelin]{Department of Mathematics, Stockholm University, SE-106 91, Stockholm,  Sweden}
\email{joeb@math.su.se}

\author[A. Oneto]{Alessandro Oneto}
\address[A. Oneto]{Department of Mathematics, Stockholm University, SE-106 91, Stockholm,  Sweden}
\email{oneto@math.su.se}

\begin{document}
\pagestyle{headings}
\maketitle

\begin{abstract}In this paper we study the class of power ideals generated by the $k^n$ forms 
$(x_0+\xi^{g_1}x_1+\ldots+\xi^{g_n}x_n)^{(k-1)d}$ where $\xi$ is a fixed primitive $k^{th}$-root of unity and 
$0\leq g_j\leq k-1$ for all $j$. For $k=2$, by using a $\ZZ_k^{n+1}$-grading on $\CC[x_0,\ldots,x_n]$, we compute the 
Hilbert series of the associated quotient rings via a simple numerical algorithm. We also conjecture the extension for $k>2$. 
Via Macaulay duality, those power ideals are related to 
schemes of fat points with support on the $k^n$ points $[1:\xi^{g_1}:\ldots:\xi^{g_n}]$ in $\PP^n$. We compute
Hilbert series, Betti numbers and Gr\"obner basis for such $0$-dimensional schemes. This explicitly determines the Hilbert
series of the power ideal for all $k$: that this agrees with our conjecture for $k>2$ is supported by several computer experiments.
\end{abstract}

\thispagestyle{empty}
\section{Introduction}

We denote by $S=\bigoplus_{i\geq 0} S_i$ the polynomial ring $\CC[x_0,\ldots,x_n]$ with the {\it standard} gradation, i.e. 
$S_d$ is the $\CC$-vector space of forms of degree $d$. 

\begin{dfnt}
 An homogeneous ideal $I\subset S$ is called a \textbf{power ideal} if $I$ is generated by some powers $L_1^{d_1},\ldots,L_m^{d_m}$ 
 of linear forms and $span(L_1,\ldots,L_m)=S_1$.
\end{dfnt}

This class of ideals received recently a considerable attention in the mathematical literature thanks to the connections with
the theories of fat points, e.g. see \cite{EI95}, \cite{GHM09}, Cox rings and box splines, see \cite{AP10} for a complete survey 
about such connections.

In this article, we want to consider a special class of power ideals depending on three positive indices and recently introduced
in connection with a Waring problem for polynomial rings, see \cite{FOS}. For any triple $(n,k,d)$ of positive integers, fixed
$\xi$ a primitive $k^{th}$-root of unity, we consider the homogeneous ideal $I_{n,k,d}$ generated by the $k^n$ powers 
$(x_0+\xi^{g_1}x_1+\ldots+\xi^{g_n}x_n)^{(k-1)d}$ where $0\leq g_j\leq k-1$ for all $j=1,\ldots,n$. We denote the quotient ring as 
$R_{n,k,d}:=\CC[x_0,\ldots,x_n]/ I_{n,k,d}$ 
and with $[R_{n,k,d}]_{j}$ its homogeneous component of degree $j$. The main results in \cite{FOS} about this class of ideals 
are the following.

\begin{teo}[\cite{FOS}, Corollary $10$]\label{MainRes}
 $[R_{n,k,d}]_{kd}=0$, i.e. for any triple $(n,k,d)$, the power ideal $I_{n,k,d}$ contains all forms of degree $kd$.
\end{teo}

As a straightforward consequence of Theorem \ref{MainRes}, the authors got the following result in terms of Waring problem 
for polynomials. 

\begin{teo}[\cite{FOS}, Theorem $4$]
A general form of degree $kd$ in $\CC[x_0,\ldots,x_n]$ is a sum of at most $k^n$ $k^{th}$-powers of forms of degree $d$.
\end{teo}

In this article, we continue the study of the family of ideals $I_{n,k,d}$ and their quotient rings $R_{n,k,d}$. The main 
goal is to determine the Hilbert series of $R_{n,k,d}$. 

In Section \ref{MultyGrad} we introduce a $\ZZ_{k}^{n+1}$-grading on $R_{n,k,d}$. It is the main tool for our first investigation
on those power ideals and, as a first consequence, we get a minimal set of generators for the ideal $I_{n,k,d}$. In Section \ref{HS1} 
we focus on the $k=2$ case. We determine the Hilbert series for the quotient ring $R_{2,n,d}$ describing a numerical 
and easily implementable algorithm. One consequence is that $[R_{2,n,d}]_{2d-1}=0$, which strengthens Theorem \ref{MainRes} 
in the $k=2$ case. In Section \ref{HS2}, we consider the $k>2$ case and we conjecture the extension of our algorithm.
 In Section \ref{pm1points}, we see how to apply our results to 
determine the Hilbert function of the schemes of fat points supported on the $k^n$ points of 
type $[1:\xi^{i_1}:\ldots:\xi^{i_n}]\in\PP^n$, where $\xi$ is a $k^{th}$ root of unity and $0\leq i_j\leq k-1$, for all $j=1,\ldots,n$. 
In particular, we get the following result and we check, with the support of a computer, that the Hilbert function provided coincides
with the conjectured algorithm in Section \ref{HS2}.

\begin{teo}\label{pmPointsThm}
 Let $I^{(d)}_k$ be the ideal of the scheme of fat points of multiplicity $d$ with support on the $k^n$ points of type 
 $[1:\xi^{g_1}:\ldots:\xi^{g_n}]\in\PP^n$ where $\xi$ is a primitive $k^{th}$-root of unity. Then, 
 we have that the Betti numbers of the quotient $S/I_k^{(d)}$ are given 
 by $\beta_{i,kd+k(i-1)}={{d+i-2}\choose{i-1}}{{d+n-1}\choose{n-i}},~~ \text{for }i=1,\ldots,n.$

Moreover, the ideal $I_{k}^{(d)}$ is generated by the degree $kd$ forms $\cG_{i_1,\ldots,i_n}:=\prod_{j=1}^n (x_j^k-x_0^k)^{i_j}$ 
for all $(i_1,\ldots,i_n)\in\NN$ with $i_1+\ldots+i_n=d$.
\end{teo}

\textbf{Acknowledgement.} The authors would like to deeply thank Ralf Fr\"oberg for his ideas and his helpful comments 
during all this project, and to express their gratitude to Boris Shapiro for the constructive meetings. The computer 
algebra software packages {\rm CoCoA} \cite{CoCoA} and {\rm Macaulay2} \cite{GS} were useful in calculations of 
many instructive examples and in the computations explained in Remark \ref{TrueConj} and Remark \ref{check}.

\section{Multicycle gradation}\label{MultyGrad}
Let $\ZZ_k=\{[0]_k,[1]_k,\ldots,[k-1]_k\}$ be the cyclic group of integers \textit{modulo} $k$. Let $\xi$ be a
primitive $k^{th}$-root of unity and observe that, for any $\nu\in\ZZ_k$, the complex number $\xi^{\nu}$ is well-defined. We will usually 
use a small abuse of notation denoting a class of integer modulo $k$ simply with its representative between $0$ and $k-1$; e.g.
when we will consider the scalar product between two vectors $\bg,{\bf h}\in\ZZ_k^{n+1}$, denoted by $\langle\bg,{\bf h}\rangle$, 
we will mean the usual scalar product considering each entry of the two vectors as the smallest positive representative of the 
corresponding class.

\vphantom{}
Consider, for each $\bg=(g_0,\ldots,g_n)\in\ZZ_k^{n+1}$, the polynomial 
$$\phi_{\bg}:=\left(\sum_{i=0}^n \xi^{g_i}x_i\right)^{D}, \text{ where }D:=(k-1)d.$$ Hence, $I_{n,k,d}$ 
is by definition the ideal generated by all $\phi_{\bg}$, with $\bg\in 0\times\ZZ^n_k.$
It is homogeneous with respect to the standard gradation, but it is also homogeneous with 
respect to the $\ZZ_k^{n+1}$-gradation we are going to define.

\vphantom{}
Consider the projection $\pi_k:\NN \longrightarrow \ZZ_k$ given by $\pi_k(n)=[n]_k$. For any vector 
$\mathbf{a}=(a_0,\ldots,a_n)\in\NN^{n+1}$, we define the \textbf{multicyclic degree} as follows. 

Given a monomial $\mathbf{x}^{\mathbf{a}}:=x_0^{a_0}\ldots x_n^{a_n}$, we set
$$\mcdeg(\mathbf{x}^{\mathbf{a}}):=\pi_k^{n+1}(\mathbf{a})=([a_0]_k,\ldots,[a_n]_k).$$

Thus, combining this multicyclic degree with the standard gradation, we get the multigradation on the 
polynomial ring $S$ given by 
$$S=\bigoplus_{i\in\NN}S_{i}=\bigoplus_{i\in\NN}\bigoplus_{\bg\in\ZZ^{n+1}_k} S_{i,\bg},\text{ where }S_{i,\bg}:=S_i\cap S_{\bg};$$
where, for any $i_1,i_2\in\NN$ and $\bg_1,\bg_2\in\ZZ_k^{n+1}$, we have that 
$$S_{i_1,\bg_1}\cdot S_{i_2,\bg_2}=S_{i_1+i_2,\bg_1+\bg_2}.$$

\begin{remark}\label{ex.1}
For $\mathbf{0}:=(0,\ldots,0)$, we get obviously that $S_{\mathbf{0}}=\CC[x_0^k,\ldots,x_n^k]$, and then, for 
any $i\in\NN$,

\centerline{$S_{i,\mathbf{0}}\neq 0$ if and only if $i=jk$ for some $j\in\NN$,} in such a case
$$\textrm{dim}_{\CC}~S_{jk,\mathbf{0}}={{n+j}\choose{n}}.$$
\end{remark}

For any arbitrary multicycle $\bg=(g_0,\ldots,g_n)\in\ZZ_k^{n+1}$, we define the \textbf{partition vector} to be  
$\partition(\bg):=\left(\#\{g_i=0\},\ldots,\#\{g_i=k-1\}\right)$
and the \textbf{weight} of $\bg$ as $\wt(\mathbf{g}):=\sum_{j=0}^n g_j.$ Clearly, the weight is non-negative and 
$$\wt(\mathbf{g})=0 \text{ if and only if } \mathbf{g}=\mathbf{0}.$$

\begin{lem}\label{Lemma1}
 Let $i\in\NN$ and $\bg\in\ZZ_{k}^{n+1}$. Then, 

 \centerline{
 $S_{i,\bg}\neq 0$ if and only if $i-\wt(\bg)=jk$, for some $j\in\NN$.}

 In such a case,  $$\mathrm{dim}_{\CC}S_{i,\bg}={{n+j}\choose{n}}.$$
\end{lem}

\begin{proof}
 Given a monomial $\bx^{\mathbf{a}}$ with $i=\deg(\bx^{\mathbf{a}})$, consider $\bg=\pi_k^{n+1}(\mathbf{a})$. 
 Hence, we have that $\bx^{\mathbf{a}-\bg}\in S_{i-\wt(\bg),\mathbf{0}}.$ Hence, 
 $$\dim_{\CC}S_{i,\bg}=\dim_{\CC}S_{i-\wt(\bg),\mathbf{0}}={{n+j}\choose{n}}.$$
\end{proof}

Now, we denote with $\G_{k,n,i}$ the set set of all multicycles satisfying the two equivalent conditions of 
Lemma \ref{Lemma1}, i.e. $$\G_{k,n,i}:=\{\mathbf{h}\in\ZZ^{n+1}_k~|~i-\wt(\bh)\in k\NN\}=\{\mathbf{h}\in\ZZ^{n+1}_k~|~S_{i,\bh}\neq 0\}.$$

Coming back to our ideal, since we can write $S_D=\bigoplus_{\bg\in\ZZ_k^{n+1}}S_{D,\bg}$, 
we can represent the generator $\phi_{\mathbf{0}}=(x_0+\ldots+x_n)^D$ of $I_{n,k,d}$ as 
$$\phi_{\mathbf{0}}=\sum_{\bg\in\ZZ_k^{n+1}}\psi_{\bg},\text{ where }\psi_{\bg}\in S_{D,\bg}.$$ 
Clearly, if $\psi_{\bg}\neq 0$ then $\bg\in\G_{k,n,D}$, but one can also check that actually 
$$\psi_{\bg}\neq 0~~~\Longleftrightarrow~~~\bg\in\G_{k,n,D}.$$ In particular, under the equivalent latter conditions, 
we have that, 
$$\psi_{\bg}=\sum_{\substack{d_0+\ldots+d_n=D \\ \pi_k^{n+1}(d_0,\ldots,d_n)=\bg}}{{D}\choose{d_0,\ldots,d_n}}\mathbf{x}^{\mathbf{d}}.$$

With the following example, we make this construction more explicit.

\begin{example}
 Consider the case $k=2,n=2,d=4$ and $\phi_{{\bf 0}}=(x_0+x_1+x_2)^{4}$. We have
$$\begin{aligned}
\psi_{(0,0,0)} & =x_0^4+6x_0^2x_1^2+6x_0^2x_2^2+x_1^4+6x_1^2x_2^2+x_2^4; \\
\psi_{(1,0,0)} & =\psi_{(0,1,0)}=\psi_{(0,0,1)}=\psi_{(1,1,1)}=0; \\
\psi_{(1,1,0)} & =4x_0^3x_1+12x_0x_1x_2^2+4x_0x_1^3; \\
\psi_{(1,0,1)} & =4x_0^3x_2+12x_0x_1^2x_2+4x_0x_2^3; \\
\psi_{(0,1,1)} & =4x_1x_2^3+12x_0^2x_1x_2+4x_1x_2^3. \\
\end{aligned}$$

We can notice that, since $(1,0,0)\not\in\G_{2,2,4}$, we already expected $\psi_{(1,0,0)}=0$, 
and similarly for $(0,1,0)$, $(0,0,1)$ and $(1,1,1)$.
\end{example}

\begin{lem}
 For any $\bg\in\ZZ_k^{n+1}$, one has $$\phi_{\bg}=\sum_{\bh\in\G_{k,n,D}}\xi^{\langle\bg,\bh\rangle}\psi_{\bh};$$ 
 conversely, $$\psi_{\bg}=k^{-n-1}\sum_{\bh\in\ZZ_k^{n+1}} \xi^{-\langle\bg,\bh\rangle}\phi_{\bh}.$$
\end{lem}
\begin{proof}
 From the definition, we can write 
 $$\begin{aligned}\phi_{\bg}&=\left(\sum_{i=0}^n\xi^{g_i}x_i\right)^D=\sum_{d_0+\ldots+d_n=D} {{D}\choose{d_0,\ldots,d_n}}
 \prod_{l=0}^n\xi^{g_ld_l}x_l^{d_l}=\\ & =\sum_{d_0+\ldots+d_n=D}{{D}\choose{d_0,\ldots,d_n}}\xi^{\langle\bg,\mathbf{d}\rangle}
 \bx^{\mathbf{d}}.\end{aligned}$$ 
 
 Now, we can consider for each $\mathbf{d}=(d_0,\ldots,d_n)$ the vector $\pi_k^{n+1}(\mathbf{d})=\bh\in\ZZ_{k}^{n+1}$. 
 Since $\xi$ is a $k^{th}$ root of unity, we have $\xi^{\langle\bg,\mathbf{d}\rangle}=\xi^{\langle\bg,\bh\rangle}$. Thus,  
 $$\phi_{\bg}=\sum_{\bh\in\G_{k,n,D}}\xi^{\langle\bg,\bh\rangle}\sum_{\substack{d_0+\ldots+d_n=D \\ \pi_k^{n+1}(\mathbf{d})=\bh}}{{D}\choose{d_0,\ldots,d_n}}
 \bx^{\mathbf{d}}=\sum_{\bh\in\G_{k,n,D}}\xi^{\langle g,h\rangle}\psi_{\bh}.$$

 For the second part of the statement, we consider the following equality which follows from the first part already proved. 
 For any $\mathbf{m}\in\ZZ_k^{n+1}$,
 
 $$\sum_{\bg\in\ZZ_k^{n+1}}\xi^{-\langle\bg,\mathbf{m}\rangle}\phi_{\bg}=\sum_{\bg\in\ZZ_k^{n+1}}\sum_{\bh\in\G_{k,n,D}}
 \xi^{\langle\bg,\bh\rangle-\langle\bg,\mathbf{m}\rangle}\psi_{\bh}.$$ 
 
 On the right hand side, we have 
 $$\begin{cases}
   \text{if }\mathbf{m}=\bh: & \sum_{\bg\in\ZZ_k^{n+1}}\psi_{\bh}=k^{n+1}\psi_{\bh}; \\
   \text{if }\mathbf{m}\neq\bh: & \sum_{\bg\in\ZZ_k^{n+1}}\xi^{\langle\bg,\bh-\mathbf{m}\rangle}\psi_{\bh}=\sum_{\bg\in\ZZ_k^{n+1}}\xi_0^{g_0}\ldots\xi_n^{g_n}\psi_{\bh}=0.
  \end{cases}$$
\end{proof}

 Hence, we have the set $\{\psi_{\bg}\}_{\bg\in\G_{k,n,D}}$ of nonzero polynomials with distinct multicyclic degree 
 and consequently linearly independent. In other words, we have proved the following proposition.

\begin{prop}
 $I_{n,k,d}$ is minimally generated by $\{\psi_{\bg}\}_{\bg\in\G_{k,n,D}}$.
\end{prop}
\begin{teo}
 The cardinality of $\G_{k,n,D}$ is given by 
 $$|\G_{k,n,D}| =$$
 $$\sum_{i\ge0}\sum_{\nu_2,\ldots,\nu_{k-1}\ge0}
 \binom{n+1}{D-ki-\sum_{j=1}^{k-1}(j-1)\nu_j}\binom{D-ki-\sum_{j=1}^{k-1}(j-1)\nu_j}{\nu_2,\ldots,\nu_{k-1},D-\sum_{j=2}^{k-1}jv_j}=$$
 $$ \sum_{i,\nu_2,\ldots,\nu_{k-1}\ge0} \binom {n+1}
{\nu_2,\ldots,\nu_{k-1},D-ki-\sum_{j=2}^{k-1}
j\nu_j, n+1-D+ki + \sum_{j=2}^{k-1}
(j-1)\nu_j}\,.$$
In particular, if  $k=2$, then
this number of generators equals $\sum_{i\ge0}\binom{n+1}{d-2i}$.
\end{teo}
\begin{proof} For any $\bg\in\G_{k,n,D}$, we can write $\psi_{\bf g}=f(x_0^k,\ldots,x_n^k)\bx^{\bg}$ where $f$ is a homogeneous 
polynomial of degree $i$ and $\partition({\bg})=(0,\nu_1,\ldots,\nu_{k-1})$. 

 In order to count the number of elements of $\G_{k,n,D}$, there are $\binom{n+1}{D-ki-\sum_{j=1}^{k-1}(j-1)\nu_j}$
 ways to choose the variables with nonzero exponent modulo $k$ and, for each such choice, there are  
 $\binom{D-ki-\sum_{j=1}^{k-1}(j-1)\nu_j}{\nu_2,\ldots,\nu_{k-1},D-\sum_{j=2}^{k-1}jv_j}$ ways to distribute the exponents.
\end{proof}
\begin{example}
For $k=4,d=3,n=2$ we get that the number of minimal generators is
$\binom{3}{0,3,0,0}+\binom{3}{0,1,2,0}+\binom{3}{1,1,0,1}+\binom{3}{2,0,1,0}+\binom{3}{0,0,1,2}=16$.
This means that the original generators $\phi_{\bf g}$ are linearly independent.
\end{example}

\begin{teo}
 If $k=2$, the generators $\{\phi_{\bg}\}_{\bg\in 0\times\ZZ_2^n}$ are linearly independent if and only if $n+1\leq d.$
\end{teo}
\begin{proof}
 $\{\psi_{\bf g}\}$ is linearly independent, and they are $\sum_{i\ge0}\binom{n+1}{d-2i}$ many.
 This sum equals $2^n$ if and only if $n+1 \le d$.
\end{proof}

\section{Hilbert function of the power ideal $I_{n,k,d}$}
In order to simplify the notation, when there will be no ambiguity, we will denote $I:=I_{n,k,d}$ and $R:=R_{n,k,d}=S/I$ with
the multicycling gradation described in the previous section $R=\bigoplus_{i\in\NN}\bigoplus_{\bg\in\ZZ_k^{n+1}}R_{i,\bg}$. 

\begin{dfnt}\label{Maps}
 For $0\leq i\leq d$ and given a vector $\bh\in\ZZ_k^{n+1}$, we define the map 
 \begin{eqnarray*}
  \mu_{i,\bh}: & D_{i,\bh}:=\bigoplus_{\bg\in\ZZ_k^{n+1}}S_{i,\bh-\bg} & \longrightarrow S_{i+D,\bh}, \\
  & (\ldots,f_{\bg},\ldots) & \longmapsto \sum_{\bg\in\ZZ_k^{n+1}}f_{\bg}\psi_{\bg}.
 \end{eqnarray*}
 given by the multiplication by each $\psi_{\bg}\in S_{D, \bg}$.
\end{dfnt}

\begin{remark}\label{Rem1}
 In order to work with relevant examples, we'll assume always that $i+D-\wt(\bh)\in k\ZZ$ in order to have $S_{i+D,\bh}\neq 0$. 
 We may also observe that, under such assumption, we have the following equivalence 
 $$i-\wt(\bh-\bg)\in k\ZZ \Longleftrightarrow D-\wt(\bg)\in k\ZZ;$$ in other words, again from the properties of this multicyclic 
 gradation explained in the previous section, we have $$S_{i,\bh-\bg}\neq 0 \Longleftrightarrow \psi_{\bg}\neq 0.$$ Thus,
 it makes sense to study the injectivity of the $\mu_{i,\bh}$'s and it will be the crucial step for our computations.
\end{remark}

\begin{lem}\label{DimLemma}
 Given $0\leq i\leq d$ and $\bh\in\ZZ_{k}^{n+1}$, if $i+D-\wt(\bh)\in k\NN$ and $\wt(\bh)\leq (k-1)(d-i)$, we have 
 $$\dim(D_{i,\bh})\leq\dim(S_{i+D,\bh});$$ with equality if $\wt(\bh)=(k-1)(d-i).$
\end{lem}

\begin{proof}
In such numerical assumptions, we have that $D_{i,\bh}$ is simply $S_i$; thus, $$\dim_{\CC}D_{i,\bh}={{n+i}\choose{n}};$$ 
moreover, we may observe that, for some integer $m\geq 0$, $$km=i+D-\wt(\bh)\geq i+D-(k-1)(d-i)=ki;$$ hence, $i+D-\wt(\bh)=k(i+j)$ 
for some $j\geq 0$ and $$\dim(S_{i+D,\bh})={{n+i+j}\choose{n}}.$$
\end{proof}

For any $0\leq i\leq d$ and $\bh\in\ZZ_{k}^{n+1}$, the image of the map $\mu_{i,\bh}$ is simply the part of multicycling degree $(i,\bh)$ 
of our ideal $I$. These maps will be the main tool in our computations regarding the Hilbert function of $I$ and its quotient ring $R$. 
By Remark \ref{Rem1} and Lemma \ref{DimLemma}, it makes sense to ask if $\mu_{i,\bh}$ is injective whenever $\wt(\bh)\leq (k-1)(d-i)$ 
and $i+D-\wt(\bh)\in k\ZZ$: in that cases, the dimension of $I_{i+D,\bh}$ in degree $i$ will be simply the dimension of 
$D_{i,\bh}=S_i$. On the other hand, again by Lemma \ref{DimLemma}, one could hope that $\mu_{i,\bh}$ is surjective in all
the other cases to get, consequently, $R_{i+D,\bh}=0$. 

This is true for $k=2$ as we are going to prove in the next section.

\subsection{The $k=2$ case}\label{HS1}
In this case, $D=(k-1)d=d$. Moreover, as we said in Remark \ref{Rem1}, we'll consider only the maps $\mu_{i,\bh}$ such that $i+d-\wt(\bh)$
is even.

\begin{lem}\label{Lem1}
 In the same notation as above, we have:
 \begin{enumerate}
  \item $\mu_{d,\bf{0}}$ is bijective;
  \item $\mu_{i,\bh}$ is injective if $\wt(\bh)\leq d-i$;
  \item $\mu_{i,\bh}$ is surjective if $\wt(\bh)\geq d-i$.
 \end{enumerate}
\end{lem}

\begin{proof}
 $(1)$ The map $\mu_{d,{\bf 0}}$ is surjective from the Theorem \ref{MainRes} and it is also injective because we are in the limit case of Lemma \ref{DimLemma}, i.e. where the dimensions of the source and the target are equal.

 $(2)$ Given a monomial $M$ with $M\in S_{d+i,\bh}$, there exists a monomial $M'$ such that $MM'\in S_{2d,{\bf 0}}$; indeed, it is enough to consider the monomial $\bx_{\bh}$ to get $\mcdeg(\bx_{\bh}M)={\bf 0}$ and then we can multiply for any monomial with the right degree to get degree equal to $2d$ and multicyclic degree equal to ${\bf 0}$. Hence, the injectivity of $\mu_{i,\bh}$ follows from $(1)$.

 $(3)$ If $\wt(\bh)=(d-i)$, we are in the limit case of Lemma \ref{DimLemma} and then, from injectivity of $\mu_{i,\bh}$, it follows also the surjectivity. Instead, the case $\wt(\bh)>(d-i)$ follows from the previous one because, given any monomial $M$ with $M\in S_{n,\bh}$ and $n-\wt(\bh)=2m$, then $M$ is a product of a monomial $M'$ with $M'\in S_{n-2m,\bh}$.
\end{proof}

We denote with $\HF(R,i)$ the {\it Hilbert function} of $R=S/I$ computed in degree $i$, i.e.
$\HF(R,i):=\dim_{\CC}(S_i)-\dim_{\CC}(I_i),$ and with $\HS(R;t)$ the {\it Hilbert series} defined as $\HS(R;t):=\sum_{i\in\NN}\HF(R,i)t^i.$

\begin{lem}\label{MainLemma}
 In the same notation as above, we have:
 \begin{enumerate}
  \item if $i<d$, $I_{i}=0$;
  \item if $i=j+d$ with $j\geq 0$, $R_{i,\bh}\neq 0$ if and only if $$\bh\in\mathcal{H}_j:=\{\bh'~|~i-\wt(\bh')\in 2\NN,~\wt(\bh')<d-j,~\wt(\bh')\leq n+1\};$$ moreover, if $\bh\in\mathcal{H}_j$, then $$\dim_{\CC}R_{i,\bh}=\dim_{\CC}S_{i,\bh}-{{n+j}\choose{n}}.$$
 \end{enumerate}
\end{lem}
\begin{proof}
 Since $I$ has generators in degree $d$, then $I_i=0$ for all $i<d$.

 Consider now $i=d+j$ for some $j\geq 0$. Since $R_{i}=\bigoplus_{\bh\in\ZZ_k^{n+1}}R_{i,\bh}$, we will focus on the dimension of each summand $R_{i,\bh}$. Fix $\bh\in\ZZ_k^{n+1}$.

 We have seen that $I=\left(\psi_{\bg}~|~\bg\in\G_{2,n,D}\right)$; hence, $I_{i,\bh}=\textrm{Im}(\mu_{j,\bh})$.

 By Lemma \ref{Lem1}, for $\wt(\bh)\geq d-j$, we know that $\mu_{j,\bh}$ is surjective and then $I_{i,\bh}=S_{i,\bh}$; consequently, $R_{i,\bh}=0$. Moreover, by Lemma \ref{Lemma1}, we need to consider only $\bh\in\ZZ_k^{n+1}$ such that $i-\wt(\bh)\in 2\NN$ otherwise $S_{i,\bh}=0$ and consequently, $R_{i,\bh}=0$. Thus, we just need to consider $\bh$ in the set $\mathcal{H}_j$ defined in the statement.

 By Lemma \ref{Lem1}, in that numerical assumptions, $\mu_{j,\bh}$ is injective and then $$\dim_{\CC}I_{i,\bh}=\sum_{\bg\in\ZZ_k^{n+1}}\dim_{\CC}S_{j,\bh-\bg}=\dim_{\CC}S_j={{n+j}\choose{n}},$$ or equivalently, $$\dim_{\CC}R_{i,\bh}=\dim_{\CC}S_{i,\bh}-{{n+j}\choose{n}}.$$
\end{proof}

\begin{teo}\label{Thm1}
 The Hilbert function of the quotient ring $R$ is given by:
 \begin{enumerate}
  \item if $i<d$, $\HF(R;i)={{n+i}\choose{n}}$;
  \item if $i=j+d$ with $j\geq 0$, $$\HF(R;i)=\sum_{\bh\in\mathcal{H}_j}\dim_{\CC}R_{i,\bh}=\sum_{\substack{h<d-j \\ i-h\in 2\NN}} {{n+1}\choose{h}}\left({{n+\frac{i-h}{2}}\choose{n}}-{{n+j}\choose{n}}\right)$$
 \end{enumerate}
\end{teo}
\begin{proof}
For $i<d$ it is trivial.

Consider $i=j+d$ with $j\geq 0$. First, we may observe that, by Lemma \ref{MainLemma}, whenever $\bh\in\mathcal{H}_j$, the dimension of $R_{i,\bh}$ depends only on the weight of $\bh$. Indeed, considering $\bh\in\mathcal{H}_j$ and denoting $h:=\wt(\bh)$, we get, by Lemma \ref{Lemma1}, $$\dim_{\CC}R_{i,\bh}=\dim_{\CC}S_{i-h,{\bf 0}}-{{n+j}\choose{n}}={{n+\frac{i-h}{2}}\choose{n}}-{{n+j}\choose{n}}.$$

To conclude our proof, we just need to observe that, fixed a weight $h$, we have exactly ${{n+1}\choose{h}}$ vectors 
$\bh\in\ZZ_{2}^{n+1}$ with such weight.
%=\sum_{\bg\in\G_{2,n,D}}\dim_{\CC}(S_{j,\bh-\bg}).$$
%
%Now, if we call $l=\wt(\bh-\bg)$, we have that, by Lemma \ref{Lemma1}, if $j-l\in 2\NN$ then $\dim_{\CC}(S_{j,\bh-\bg})=\dim_{\CC}(S_{j-l,\bf{0}})$; moreover, fixed a weight $l$, we have ${{n+1}\choose{l}}$ possible choices of $\bg$ to get $l=\wt(\bh-\bg)$. Thus, we get the following formula $$\HF(R;i,\bh)=\dim_{\CC}(S_{i,\bh})-\sum_{\substack{j-l\in 2\NN \\ l\geq 0}}{{n+1}\choose{l}}\dim_{\CC}(S_{j-l,\bf{0}}).$$ Actually, we can even simplify such formula a little more. The sum can be rewritten as follows. Say $div(j,2)=m$, i.e. $j=2m$ or $j=2m+1$, depending on the parity.
%$$\begin{aligned}\sum_{\substack{j-l\in 2\NN \\ l\geq 0}}{{n+1}\choose{l}}\dim_{\CC}(S_{j-l,\bf{0}})&\underset{div(l,2)=h}{=}\sum_{h=0}^m{{n+1}\choose{l}}\dim_{\CC}(S_{2m-2h,\bf{0}})= \\ &=\sum_{h=0}^m{{n+1}\choose{l}}{{n+m-h}\choose{n}}.\end{aligned}$$
%
%Now, since every monomial $M$ in $n+1$ variables and degree $j=2m$ (resp. $j=2m+1$) can be written as $M=N^2\cdot H$ where $N$ is a monomial of degree $m-h$ and $H$ is the product of $2h$ (resp. $2h+1$) distinct variables, we get easily that the last sum above is ${{n+j}\choose{n}}$.
\end{proof}

\begin{corollary}\label{Cor1}
 $R_{2d-1}=0$.
\end{corollary}
\begin{proof}
 $R_{2d-1,\bh}\neq 0$ if and only if $\wt(\bh)$ is odd and $\wt(\bh)<1$, so never.
\end{proof}

%\begin{remark}
%The result in Theorem \ref{Thm1} can be easily implemented by using, for example, CoCoA programming language as follows.
%
%\vphantom{}
%\begin{verbatim}
%-- Input parameters n and d;
%N:= ;
%D:= ;
%
%-- HF will be the vector representing the Hilbert Function;
%HF:=[];
%
%-- Computation of the Hilbert Function:
%-- 1) in degree <d: 
%Foreach L In 0..D-1 Do
%  Append(HF,Bin(N+L,N));
%EndForeach;
%
%-- 2) in degree =d,..,2d-1:
%ForEach J In 0..D-1 Do
%  I:=D+J;
%  H:=[];
%  M:=0;
%	
%  ForEach K In 0..I Do
%    If Mod(I-K,2)=0 Then
%      If K<D-J Then
%        If K<N+2 Then
%          Append(H,K);
%          M:=M+1;
%        EndIf;
%      EndIf;
%    EndIf;
%  EndForeach;
%
%  HH:=0;
%  If M>0 Then
%    Foreach K In 1..M Do
%      HH:=HH+(Bin(N+1,H[K]))*(Bin(N+Div(I-H[K],2),N)-Bin(N+J,N));
%    EndForeach;
%  EndIf;
%
%  Append(HF,HH);
%EndForeach;
%
%HF;
%\end{verbatim}
%
%Since completely numerical, our algorithm works fast even with very large values of $n$ and $d$, e.g. $n, d \sim 300$; something that, of course, the classical algorithm implemented by involving the computation of Gr\"obner basis cannot do.
%\end{remark}

In the following example, we explicit our algorithm in a particular case in order to help the reader in the comprehension of the theorem. 

\begin{example}\label{Example}
 Let's take $n+1=4$, i.e. $S=\CC[x_0,\ldots,x_3]$ and $d=5$. We compute the Hilbert function of the quotient $R=S/I_{2,3,5}$ where $$I_{2,3,5}=\left((x_0\pm x_1\pm x_2\pm x_3)^5\right).$$

\vphantom{}
 For $i<5$, we have $$\HF(R;i)={{3+i}\choose{3}}.$$

\vphantom{}
 For $i=5$ ($j=0$), we have that $\mathcal{H}_0=\{\bh~|~\wt(\bh)=1,3\}$, hence $$\begin{aligned}\HF(R;5)&=\sum_{\wt(\bh)=1}\dim_{\CC}R_{5,\bh}+\sum_{\wt(\bh)=3}\dim_{\CC}R_{5,\bh}= \\ &={{4}\choose{1}}(\dim_{\CC}(S_{4,\bf{0}})-1)+{{4}\choose{3}}(\dim_{\CC}(S_{2,\bf{0}})-1)=\\ &=4(10-1)+4(4-1)=36+12=48.\end{aligned}$$

\vphantom{}
 For $i=6$ ($j=1$), we have that $\mathcal{H}_1=\{\bh~|~\wt(\bh)=0,2\}$, hence $$\begin{aligned}\HF(R;6)&=\dim_{\CC}R_{6,\mathbf{0}}+\sum_{\wt(\bh)=2}\dim_{\CC}R_{6,\bh}= \\ &=(\dim_{\CC}(S_{6,\bf{0}})-4)+{{4}\choose{2}}(\dim_{\CC}(S_{4,\bf{0}})-4)= \\ &=(20-4)+6(10-4)=16+36=52.\end{aligned}$$

\vphantom{}
 For $i=7$ ($j=2$), we have that $\mathcal{H}_2=\{\bh~|~\wt(\bh)=1\}$, hence $$\HF(R;7)=\sum_{\wt(\bh)=1}\dim_{\CC}R_{7,\bh}={{4}\choose{1}}(\dim_{\CC}(S_{6,\bf{0}})-10)=4(20-10)=40.$$

\vphantom{}
 For $i=8$ ($j=3$), we have that $\mathcal{H}_3=\{\mathbf{0}\}$, hence $$\HF(R;8)=\dim_{\CC}R_{8,\mathbf{0}}=\dim_{\CC}(S_{8,\bf{0}})-20=35-20=15.$$

\vphantom{}
 For $i\geq 9$ ($j\geq 4$), we can easily see that $\mathcal{H}_j=\emptyset$. Thus, the Hilbert function is 

 \begin{center}
  \begin{tabular}{c | c c c c c c c c c c}
   $i$ & 0 & 1 & 2 & 3 & 4 & 5 & 6 & 7 & 8 & 9 \\
    \hline{}
   $\HF(R;i)$ & 1 & 4 & 10 & 20 & 35 & 48 & 52 & 40 & 15 & - \\
  \end{tabular}
 \end{center}
\end{example}

\vphantom{}
With the following theorem, we are going to work on our result in order to compute more explicitly the Hilbert series in cases with small number of variables.

\begin{teo}\label{Thm2}
 The Hilbert series of $R_{2,1,d}$ is given by $(1-2t^d+t^{2d})/(1-t)^2$.

 The Hilbert series of $R_{2,2,d}$, for $d\geq 2$ is given by 
 $$\begin{aligned} \HS(R_{2,2,d};t)&=\frac{\left(1-4t^d+dt^{2d-1}+3t^{2d}-dt^{2d+1}\right)}{(1-t)^3}= \\ &=\sum_{i=0}^{d-1} {{i+2}\choose{2}}t^i+\sum_{i=0}^{d-2}\left({{d+i+2}\choose{2}}-4{{i+2}\choose{2}}\right)t^{d+i}. \end{aligned}$$

 The Hilbert series of $R_{2,3,d}$, for $d\geq 3$ is given by 
 $$\begin{aligned} \HS&(R_{2,3,d};t)= \\ &=\frac{\left(1-8t^d+{{d}\choose{2}}t^{2d-2}+4dt^{2d-1}-(d^2-7)t^{2d}-4dt^{2d+1}+{{d+1}\choose{2}}t^{2d+2}\right)}{(1-t)^4}= \\ &=\sum_{i=0}^{d-1} {{i+3}\choose{3}}t^i+\sum_{i=0}^{d-3}\left({{d+i+3}\choose{3}}-8{{i+3}\choose{3}}\right)t^{d+i} +{{d+1}\choose{2}}t^{2d-2}.\end{aligned}$$
\end{teo}

\begin{proof}
 Case $n+1=2$. Simply, we have a complete intersection and it follows that the Hilbert series is $(1-2t^d+t^{2d})/(1-t)^2$.

\vphantom{}
 Case $n+1=3$. From Lemma \ref{Lem1}, we have that $[I_{2,2,d}]_{d+j}=S_{j}[I_{2,2,d}]_d$ for any $0\leq j\leq d-3$ since 
 $\wt(\bh)\leq d-3$ for 
 all possible $\bh$. Since $2d-2$ is even, we get that $\wt(\bh)$ should be even and then, $\wt(\bh)\leq 2=d-(d-2)$; thus, 
 we get injectivity also in this degree. Now, from Theorem \ref{Thm1}, we get that 
 $\dim_{\CC}([R_{2,2,d}]_{d+j})=\dim_{\CC}(S_{d+j})-\#(\mathcal{H}_j)\cdot {{n+j}\choose{n}}$. 

 In our numerical assumption, it is clear that, for $0\leq i\leq d-3$, $\mathcal{H}_i$ is exactly the half of all 
 possible vectors in $\ZZ_{2}^{n+1}$, i.e. $\#(\mathcal{H}_i)=2^n$; hence, 
 $$\HS(R_{2,2,d};t)=\sum_{i=0}^{d-1} {{i+2}\choose{2}}t^i+\sum_{i=0}^{d-2}\left({{d+i+2}\choose{2}}-4{{i+2}\choose{2}}\right)t^{d+i}.$$

 A simple calculation shows that $(1-t)^3\HS(R_{2,2,d};t)=(1-4t^d+dt^{2d-1}+3t^{2d}-dt^{2d+1}).$

\vphantom{}
 Case $n+1=4$. From Lemma \ref{Lem1}, since $\wt(\bh)\leq 4$ for all possible $\bh$, we get that $[I_{2,3,d}]_{d+i}=S_i[I_{2,3,d}]_d$ 
 for all 
 $0\leq i\leq d-4$. Moreover, since $2d-3$ is odd, we get that $\wt(\bh)$ should be odd and consequently $\wt(\bh)\leq 3=d-(d-3)$; 
 hence, we have injectivity also in this degree. Moreover, for all $0\leq i\leq d-3$, we get that $\mathcal{H}_i$ is half of all 
 possible vectors in $\ZZ_{2}^{n+1}$, i.e. $\mathcal{H}_i$ has cardinality equal to $2^n$.

 Now, we just miss to compute the dimension of $[R_{2,3,d}]_{2d-2}$. By definition, the vectors $\bh\in\mathcal{H}_{d-2}$ 
 have to be odd, since $2d-2$ is odd, and to satisfy the condition $\wt(\bh)<2$; thus, we get only $\bh={\bf 0}$ and 
 $\#(\mathcal{H}_{d-2})=1$. Thus, by Theorem \ref{Thm1}, 
 $$\begin{aligned}\dim_{\CC}([R_{2,3,d}]_{2d-2})&=\dim_{\CC}([R_{2,3,d}]_{2d-2,{\bf 0}})=\dim_{\CC}(S_{2d-2,{\bf 0}})-{{3+d-2}\choose{3}}=\\ 
 &={{d+2}\choose{3}}-{{d+1}\choose{3}}={{d+1}\choose{2}}.\end{aligned}$$

 Putting together our last observations, we get $$\HS(R_{2,3,d};t)=\sum_{i=0}^{d-1} {{i+3}\choose{3}}t^i+\sum_{i=0}^{d-3}\left({{d+i+3}\choose{3}}-8{{i+3}\choose{3}}\right)t^{d+i}+{{d+1}\choose{2}}t^{2d-2}.$$ 
 A simple calculation shows that $$\begin{aligned}(1&-t)^4\HS(R_{2,3,d};t)= \\ &=\left(1-8t^d+{{d}\choose{2}}t^{2d-2}+4dt^{2d-1}-(d^2-7)t^{2d}-4dt^{2d+1}+{{d+1}\choose{2}}t^{2d+2}\right).\end{aligned}$$
\end{proof}

\begin{remark}\label{Rem1}
 From the proof of Theorem \ref{Thm2}, we can say something more also about the Hilbert series of $R_{2,n,d}$ even for more variables. 

 Assuming $d\geq n$, by using the same ideas as in the theorem above, we get that for all $0\leq j\leq d-n$, the 
 $(d+j)^{th}$-coefficient of our Hilbert series is equal to $$\HF(R_{2,n,d};d+j)={{n+d+j}\choose{n}}-2^n{{n+j}\choose{n}}.$$

 Moreover, we get that, for any $d\geq 2$, $\mathcal{H}_{d-2}=\{{\bf 0}\}$ and consequently, 
 $$\begin{aligned}
 \HF(R_{2,n,d}&;2d-2)=\dim_{\CC}([R_{2,n,d}]_{2d-2})=\dim_{\CC}([R_{2,n,d}]_{2d-2,{\bf 0}})= \\  
 &=\dim_{\CC}(S_{2d-2,{\bf 0}})-{{n+d-2}\choose{n}}={{n+d-1}\choose{n}}-{{n+d-2}\choose{n}}=\\  
 &={{n+d-2}\choose{n-1}}.
 \end{aligned}$$

 Similarly, we have that, for any $d\geq 3$, $\mathcal{H}_{d-3}=\{h\in\ZZ_k^{n+1} ~|~\wt(h)=1 \}$, thus
$$\begin{aligned}\HF(R_{2,n,d};2d-3)&=\dim_{\CC}([R_{2,n,d}]_{2d-3})=\sum_{\wt(h)=1} \dim_{\CC}([R_{2,n,d}]_{2d-3,h})= \\ 
&=(n+1)\left[\dim_{\CC}(S_{2d-2,{\bf 0}})-{{n+d-2}\choose{n}}\right]=\\ &=(n+1){{n+d-2}\choose{n-1}}.\end{aligned}$$
\end{remark}

\begin{conj}
 $R_{2,n,d}$ is \textit{level} algebra, i.e. ${\rm Soc}(R_{2,n,d})=[R_{2,n,d}]_{2d-2}$. 
\end{conj}

If so, from Remark \ref{Rem1}, we would have that ${\rm Soc}(R_{2,d,n})$ has dimension ${{n+d-2}\choose{n-1}}$.

\subsection{The $k>2$ case.}\label{HS2}

We would like to generalize our results for the cases $k>2$. Inspired by Lemma \ref{DimLemma}, we conjecture the 
following behavior of the maps $\mu_{i,\bh}$.

\begin{conj}\label{Conj1}
  In the same notation as Definition \ref{Maps}, we have
 \begin{enumerate}
  \item $\mu_{i,\bh}$ is injective if $\wt(\bh)\leq (k-1)(d-i)$;
  \item $\mu_{i,\bh}$ is surjective if $\wt(\bh)\geq (k-1)(d-i)$.
 \end{enumerate}
\end{conj}

\noindent Following the same ideas as Lemma \ref{MainLemma}, from Conjecture \ref{Conj1} we would get the following results.

\begin{conj}\label{Conj2}
 In the same notation as above, we have

if $i=j+D$ with $j\geq 0$, $R_{i,\bh}\neq 0$ if and only if 
$$\bh\in\mathcal{H}_j:=\{\bh'~|~i-\wt(\bh')\in k\NN,~\wt(\bh')<d-j,~\wt(\bh')\leq(k-1)(n+1)\};$$ moreover, if $\bh\in\mathcal{H}_j$, 
then $$\dim_{\CC}R_{i,\bh}=\dim_{\CC}(S_{i,\bh})-{{n+j}\choose{n}}.$$

%
%\sum_{\substack{j-l\in k\NN \\ l\geq 0}}N_l\dim_{\CC}(S_{j-l,\bf{0}}),$$ where, for any $l\geq 0$, $N_l:=\#\{\bh\in\ZZ_k^{n+1}~|~\wt(\bh)=l\}$.

\end{conj}

\begin{prop}
 Conjecture \ref{Conj1} $\Longrightarrow$ Conjecture \ref{Conj2}.
\end{prop}
\begin{proof}
 Follow the proof of Theorem \ref{Thm1}.
\end{proof}

\begin{remark}\label{AlgConj}
 From these conjectures, it would follow a direct generalization of the algorithm described in Example \ref{Example} to compute the Hilbert function 
of the quotient rings $R$. Trivially, we already know that, for $i<D$, since the ideal $I$ has generators only in degree $D$, $$\HF(R;i)={{n+i}\choose{n}}.$$

 For the cases $i=D+j$ with $j\geq 0$, from Conjecture \ref{Conj2}, we would have 
$$\HF(R;i)=\sum_{\substack{h<(k-1)(d-j) \\ i-h\in k\NN}} N_h\left({{n+\frac{i-h}{k}}\choose{n}}-{{n+j}\choose{n}}\right);$$ 
where $N_h$ is simply the number of vectors $\bh\in\ZZ_{k}^{n+1}$ of weight $\wt(\bh)=h$. In order 
to compute the numbers $N_h$ we may look at the following formula, 
$$
\sum_{h=0}^{(k-1)(n+1)}N_hx^h=(1+x+\ldots+x^{k-1})^{n+1}=\left(\frac{1-x^k}{1-x}\right)^{n+1};
$$
from there, expanding the right hand side, we get, for all $h=0,\ldots,(k-1)(n+1)$, 
$$N_h=\sum_{s=0}^{\lfloor\frac{h}{k}\rfloor}(-1)^s{{n+1}\choose{s}}{{n+h-ks}\choose{n}}.$$
\end{remark}

\begin{remark}\label{kd-1 Rem}
From the conjectures, we would get also the extension of Corollary \ref{Cor1} in the $k> 2$ case, i.e. $$[R_{k,n,d}]_{kd-1}=0.$$ Indeed, 
with the same notation as above, let's take $j=d-1$. Thus, to compute the Hilbert function of the quotient in position $kd-1$ 
we should compute the set $\mathcal{H}_{d-1}$, i.e. the set of $\bh\in\ZZ_{k}^{n+1}$ satisfying the following conditions:
$$kd-1-\wt(\bh)\in k\ZZ,~~\wt(\bh)<(k-1)(d-d+1)=k-1.$$ From the first condition, we get that $\wt(\bh)\in (k-1)+k\ZZ_{\geq 0}$
which is clearly in contradiction with the second condition above. Thus, $\mathcal{H}_{d-1}$ is empty and $\HF(R;kd-1)=0$.
\end{remark}

%\subsection*{The algorithm}
% Let's sketch the algorithm in pseudo-code. After that we'll give an example to better understand each step.
%\begin{verbatim}
% input k,d,n;
% D:=(k-1)d;
%
% % compute vector N of length (k-1)(n+1)+1;
% for l=0..(k-1)(n+1)
%     N[l]:=#{Young diagram with l boxes contained in 
%             (k-1)x(n+1) box, counting permutations};
%
% % compute Hilbert function
% for i=0..D-1
%     HF[i]:=binom_coeff(n+i,n);
%
% for i=D..kd
%     j:=i-D;
%
%     % compute vector H of "possible" weights
%     for h=0..(k-1)(n+1)
%         if mod(i-h,k)=0 & h<(k-1)(d-j) then
%             append(H,h);
%
%     l1:=length(H);
%
%     m:=div(i-h,k);
%	
%     for s=0..j
%         if mod(j-s,k)=0 then
%             append(L,s);
%
%     l2:=length(L)
%	
%     HF[i]:=0;
%	
%     for n=1..l1
%         R[n]:=binom_coeff(n+m,n);
%         for m=1..l2
%             R[n]:=R[n]-N[L[n]]*binom_coeff(n+div(j-L[n],k),n);
%             HF[i]:=HF[i]+N[H[n]]*R[n];
%	
%\end{verbatim}

\begin{example}
 Let's give one explicit example of the computations in order to clarify the algorithm.

 We consider the following parameters: $k=4,~n=2,~d=8$. Thus we have $D=24$.
 Let's compute, for example, the Hilbert function of the corresponding quotient ring in degree $i=28$, i.e. $j=4$. 
 Via the support of a computer algebra software, as CoCoA5 \cite{CoCoA} or Macaulay2 \cite{GS} and the implemented functions involving Gr\"obner basis, one can see that 
 $$\HF(R;28)=195.$$

 Let's apply our algorithm to compute the same number. First, we need to write down the vector $N$ where, for 
 $l=0\ldots(k-1)(n+1)$, $N_l:=\#\{\bh\in\ZZ_{k}^{n+1}~|~\wt(\bh)=l\}$. In our numerical assumptions we have 
 $$N=(N_0,\ldots,N_9)=(1,3,6,10,12,12,10,6,3,1).$$

% \begin{figure}[h]
% \begin{center}
% \includegraphics[scale=0.20]{N3.pdf}
% % \caption{Example: computation of $N_3=10$.}
% \end{center}
% \end{figure}

 Now, we need to compute the vector $H$ where we store all the possible weights for the vectors $\bh\in\mathcal{H}_4$, i.e.
 all the number $0\leq h\leq9$ s.t. the following numerical conditions hold, $$28-h\in 4\ZZ,~ h<(k-1)(d-j)=12;$$ thus, 
 $H=(H_0,H_1,H_2)=(0,4,8)$. 
 Hence, we can finally compute $\HF(R;28,\bh)$ for each $\bh\in\mathcal{H}_4$. From our formula, it is clear that such 
 numbers depend only on the weight of $\bh$; thus, we just need to consider each single element in the vector $H$.

 Assume $\wt(\bh)=0$. We get, 
 $$\begin{aligned}R_{0}:=\HF(R;28,{\bf 0})=\dim_{\CC}S_{28,{\bf 0}}-{{n+j}\choose{n}}=36-15=21;\end{aligned}$$
 Similarly, we get: if $\wt(\bh)=4$, 
 $$\begin{aligned}R_{4}:=\HF(R;28,{\bf h})=\dim_{\CC}S_{24,{\bf 0}}-{{n+j}\choose{n}}=28-15=13;\end{aligned}$$
 and, if $\wt(\bh)=8$, 
 $$\begin{aligned}R_8:=\HF(R;28,{\bf h})=\dim_{\CC}S_{20,{\bf 0}}-{{n+j}\choose{n}}=21-15=6.\end{aligned}$$

 Now, we are able to compute the Hilbert function in degree $28$.
\begin{align*}\HF(R;28)&=N_{H_0}R_{H_0}+N_{H_1}R_{H_1}+N_{H_2}R_{H_2}= \\ 
&=21+12\cdot 13+3\cdot 6=21+156+18=195.\end{align*}
\end{example}

\subsection{The algorithm.}\label{algorithm}
In this section we want to show our algorithm implemented by using CoCoA5 programming language, see \cite{CoCoA}. 
As we have seen in the previous section, in the case $k>2$, the algorithm is just conjectured. However, as we will see in Section \ref{check},
we made several computer experiments supporting our conjectures. 
Here is the CoCoA5 script of our algorithm based on Theorem \ref{Thm1} and Remark \ref{AlgConj}.

\begin{verbatim}
-- 1) Input parameters K, N, D;
 K := ;
 N := ;
 D := ;
 DD :=(K-1)*D;

-- HF will be the vector representing the Hilbert function 
-- of the quotient ring;
 HF :=[];

-- 2) Input vector NN where NN[I] counts the number of vectors 
-- in ZZ^{n+1} modulo K of weight I;
 Foreach H In 0..((N+1)*(K-1)) Do 
   M := 0;
   Foreach S In 0..(Div(H,K)) Do
     M := M+(-1)^S*Bin(N+1,S)*Bin(N+H-K*S,N);
   EndForeach;
   Append(Ref NN,M);
 EndForeach;

-- 3) Compute the Hilbert Function:
-- in degree <DD: 
 Foreach L In 0..(DD-1) Do
   Append(Ref HF,Bin(N+L,N));
 EndForeach;

-- in degree =DD,..,K*D-1:
 Foreach J In 0..(D-2) Do
   I:=DD+J;
   H:=[];
   M:=0;
	
   Foreach S In 0..I Do
     If Mod(I-S,K)=0 Then
       If S<(K-1)*(D-J) Then
         If S<(K-1)*(N+1)+1 Then
           Append(Ref H,S);
           M:=M+1;
         EndIf;
       EndIf;
     EndIf;
   EndForeach;

   HH:=0;
   If M>0 Then
     Foreach S In 1..M Do
       HH:=HH+NN[H[S]+1]*(Bin(N+Div(I-H[S],K),N)-Bin(N+J,N));
     EndForeach;
   EndIf;

   Append(Ref HF,HH);
 EndForeach;

-- 4) Print the Hilbert function:
 HF;
\end{verbatim}

\begin{remark}\label{TrueConj}
In the $k=2$ case, our algorithm, which is proved to be true by Theorem \ref{Thm1}, works very fast even with large values of 
$n$ and $d$, 
e.g. $n, d \sim 300$; cases that the computer algebra softwares, by involving the computation of Gr\"obner basis, cannot do in a reasonable amount of time and memory.

 As regards the $k>2$ case, with the support of computer algebra software Macaulay2 and
 its implemented function to compute Hilbert series of quotient rings, we have checked that our numerical algorithm 
 produces the right Hilbert function for two and three variables for low $k$ and $d$. Moreover,
 in Section \ref{pm1points}, we study the schemes of fat points related to our power ideals and our results 
 on their Hilbert series, will support Conjecture \ref{Conj2} in much more cases. With the support of the computer algebra software CoCoA5, we have checked that the conjectured algorithm gives the correct Hilbert function for all 
 $$n+1=3,4,5,~k=3,4,5\text{ and }d\leq 150.$$
\end{remark}

\section{Hilbert function of $\xi$-points in $\PP^n$}\label{pm1points}
As we said in the introduction, there is a close connection between power ideals and many different theories of mathematics. 
In this section, we want to see how our results can give important informations on particular arrangement of {\it fat points} in
projective spaces. We will consider schemes of fat points with support on the $k^{n}$ points of type 
$\left[1:\xi^{g_1}:\ldots:\xi^{g_n}\right]\in\PP^n$
where $\xi$ is a fixed primitive $k^{th}$-root of unity and $0\leq g_i\leq k-1$ for all $i=1,\ldots,n.$ 
Thanks to our results in Section \ref{HS1} and Section \ref{HS2}, we have been able to completely understand these schemes of 
fat points in terms of generators, Hilbert series and Betti numbers.

\smallskip{}
For any point $P$ in the projective space $\PP^n$ we can associate the prime ideal $\wp \subset \CC[x_0,\ldots,x_n]$ which consists
of the ideal of all homogeneous polynomials vanishing at the point $P$, namely all the hypersurfaces passing through the point $P$.

A {\bf fat point} supported at $P$ is the non-reduced $0$-dimensional scheme associated to some power $\wp^d$ of the prime ideal. 
Such scheme is usually denoted with $dP$ and consists of all homogeneous polynomials such that all differentials of degree $\leq d-1$ 
vanish at the point $P$. From a geometrical point of view, it is the ideal of all hypersurfaces of $\PP^n$ which are singular at $P$
with multiplicity $d$.

In general, we can consider a {\bf scheme of fat points} $\XX=dP_1+\ldots+dP_g$ as the $0$-dimensional scheme in $\PP^n$ associated 
to the ideal $I^{(d)}=\wp_1^d\cap\ldots\cap\wp_g^d$ where $\wp_i$ is the ideal associated to the point $P_i$ 
for all $i=1,\ldots,g$, respectively.
Such ideal is, from an algebraic point of view, the $d^{th}${\it -symbolic power} if the ideal $I=\wp_1\cap\ldots\cap\wp_g$. 

\smallskip{}
The relation between power ideals and fat points is given by the {\it Macaulay duality} or {\it Apolarity Lemma}. For all 
positive integer $d$, we can consider the power ideal
$I_d=(L_1^d,\ldots,L_g^d)\subset S=\CC[x_0,\ldots,x_n]$ where $L_i=a_0^{(i)}x_0+\ldots+a_n^{(i)}x_n$, for all $i=1,\ldots,g$.
We can associate to each linear form $L_i$ the projective points $P_i=[a_0^{(i)}:\ldots:a_n^{(i)}]\in\PP^n$ and its associated 
prime ideal $\wp_i$. Let $I=\wp_1\cap\ldots\cap\wp_g$. 

The Macaulay duality connects the Hilbert function of the quotients $R_d=S/I_d$ with the Hilbert function of the schemes of fat 
points associated to the symbolic powers of $I$, see \cite{EI95} or \cite{Ger96}.

\begin{teo}[{\sc Macaulay duality}]\label{MacaulayDuality}
 For all $m\geq d$, we have that 
 \begin{equation*}
  \HF(I^{(d)},m)=\HF(R_{m-d+1},m).
 \end{equation*}
\end{teo}

\subsection{The $k=2$ case.}
We begin by considering our class of power ideals in the $k=2$ case, where the generators of the ideal $I_d$ are the $d^{th}$-powers of the 
$2^n$ linear forms of type $L=x_0\pm x_1\pm\ldots\pm x_n$. In Section \ref{HS1}, we have described a easy algorithm to compute the
Hilbert function of the quotient rings $R_d=S/I_d$, thus, via Macaulay duality, we can apply our computations to get the
Hilbert function of schemes of fat points supported at all $(\pm1)${\it -points} of $\PP^n$, namely the $2^n$ points of the type 
$[1:\pm 1:\ldots\pm 1]$. We'll see later that the results for these arrangement of points can be directly extended to the $k>2$ case.

\begin{prop}\label{Prop1}
 Let $I^{(d)}$ be the ideal associated to the scheme of $d$-fat points supported on the $(\pm1)$-points of $\PP^n$. Then, 
 \begin{equation*}
  \HF(S/I^{(d)},m)=
   \begin{cases}
    {{n+m}\choose{n}} & \text{for }m\leq 2d-1 \\
    {{n+2d}\choose{n}}-{{d+n-1}\choose{n-1}} & \text{for }m=2d \\
    {{n+2d+1}\choose{n}}-(n+1){{d+n-1}\choose{n-1}} & \text{for }m=2d+1 \\
    2^n{{n+d-1}\choose{n}} & \text{for }m\geq 2d+n-2 \\
   \end{cases}
 \end{equation*}
\end{prop}

\begin{proof}
 By Corollary \ref{Cor1}, we know that $\HF(R_{m-d+1},m)=0$ for all $m$ satisfying the inequality 
 $m\geq 2(m-d+1)-1$ or, equivalently $m\leq 2d-1;$
 moreover, by Remark \ref{Rem1}, we have that $\HF(R_{d+1},2d)={{n+d-1}\choose{n-1}}$, 
 $\HF(R_{d+2},2d+1)=(n+1){{n+d-1}\choose{n-1}}$ and $\HF(R_{m-d+1},m)={{n+m}\choose{n}}-2^n{{n+d-1}\choose{n}}$ for $m\leq 2(m-d+1)-n$, or equivalently, $m\geq 2d+n-2$. By Macaulay duality, we are done.
\end{proof}

\begin{remark}
 Such result tell us that the ideal $I^{(d)}$ is generated in degree $\geq 2d$ and, in particular, with ${{d+n-1}\choose{n-1}}$ generators in degree $2d$. Thanks to the geometrical meaning of the symbolic power $I^{(d)}$, we can easily find such generators.
\end{remark}

 We may observe that we have exactly $n$ pairs of hyperplanes which split our $2^n$ points. Namely, for any variable 
 except $x_n$, we can consider the hyperplanes 
 $$H_i^+ = \{ x_i+x_n=0 \}\text{   and   }H_i^-=\{x_i-x_n=0\}, \text{  for all }i=0,\ldots,n-1.$$ 
 
 It is clear that, for all $i$, half of our $(\pm1)-$points lie on $H_i^+$ and half on $H_i^-$. Consequently,
 we have $n$ quadrics passing through our points exactly once, i.e. $\Q_i=H_i^+H_i^-=x_i^2-x_n^2$, for all $i=0,\ldots,n-1.$
 
 Now, we want to find the generators of $I^{(d)}$, hence we want to find hypersurfaces passing through our points with multiplicity $d$
 and we can consider, for example, all the {\it monomials} of degree $d$ constructed with these quadrics $\Q_0,\ldots,\Q_{n-1}$, i.e.
 the degree $2d$ forms $$\cG_1:=\Q_0^d,~\cG_2:=\Q_0^{d-1}\Q_1,~\cG_3:=\Q_0^{d-1}\Q_2,\ldots,\cG_N:=\Q_{n-1}^d,$$ 
 where $N={{n+d-1}\choose{n-1}}$. We can actually prove that they generate the part of degree $2d$ of $I^{(d)}$ as a $\CC$-vector space.
 Since the number of $\cG_i$'s is equal to the dimension of $[I^{(d)}]_{2d}$ computed in Proposition \ref{Prop1}, it is enough to prove 
 the following statement.
 
 \vphantom{}
 {\it Claim.} The $\cG_i$'s are linearly independent over $\CC$.
 
 \smallskip{}
 \begin{proof}[Proof of the Claim]
  We prove it by double induction over the number of variables $n$ and the degree $d$. For two variables, i.e. $n=1$, we have that 
  the dimension of $[I^{(d)}]_{2d}$
  is equal to $1$ for all $d$ and then, $\cG_1=\Q_0^d$ is the unique generator. For $n>1$, we consider first the $d=1$ case.
  Assume to have a linear combination $$\alpha_0\Q_0+\ldots+\alpha_{n-1}\Q_{n-1}=\alpha_0(x_0^2-x_n^2)+\ldots+\alpha_{n-1}(x_{n-1}^2-x_n^2)=0.$$
  Specializing on the hyperplane $H_0^-=\{x_0=x_n\}$, we reduce the linear combination in one variable less and, by induction, we
  have $\alpha_i=0$ for all $i=1,\ldots,n-1$; consequently, also $\alpha_0=0$.
  
  Assume to have a linear combination for $d\geq 2$, namely 
  \begin{align*}
   \alpha_1\cG_1&+\alpha_2\cG_2+\ldots+\alpha_N\cG_N= \\
   &=\alpha_1(x_0^2-x_n^2)^d+\alpha_2(x_0^2-x_n^2)^{d-1}(x_1^2-x_n^2)+\ldots+\alpha_N(x_0^2-x_n^2)^d=0.
  \end{align*}
  
  By specializing again on the hyperplane $H_0^-=\{x_0=x_n\}$, we get a linear combination in the same degree but with one variable 
  less and, by induction over $n$, we have that $\alpha_i=0$ for all $i$ where the definition $\G_i$ doesn't involve $(x_0^2-x_n^2)^d$.
  
  Thus, we remain with a linear combination of type 
  $$(x_0^2-x_n^2)\left[ \alpha_0\Q_0^{d-1}+\alpha_1\Q_0^{d-2}\Q_1+\ldots+\alpha_m\Q_{n-1}^{d-1}\right]=0;$$ 
  by induction over $d$, we are done.
 \end{proof}
 
 Hence, we can consider the ideal $J_d=(x_0^2-x_n^2,\ldots,x_{n-1}^2-x_n^2)^d$. It is clearly contained in $I^{(d)}$ but, {\it a priori}, it could be 
 smaller. 
 
 In order to show that the equality holds and that $I^{(d)}$ is minimally generated by the $\G_i$'s, we start by studying the Hilbert
 series of the ideal $J_d$.
 
 \begin{lem}\label{HS&Betti}
  Let $T_d=\CC[x_0,\ldots,x_n]/J_d$, where $J_d=(x_0^2-x_n^2,\ldots,x_{n-1}^2-x_n^2)^d$, then the Hilbert series is
  $$\HS(T_d;t)=\frac{1+\sum_{i=1}^n(-1)^i\beta_it^{2d+2(i-1)}}{(1-t)^{n+1}},$$
  where $\beta_i:=\beta_{i,2d+2(i-1)}={{d+i-2}\choose{i-1}}{{d+n-1}\choose{n-i}},$ for all $i=1,\ldots,n$, and the multiplicity is 
  $e(T_{d})=2^n{{d+n-1}\choose{n}}.$
 \end{lem}
 
 \begin{proof}
 The quotient $T_d$ is a $1$-dimensional Cohen-Macaulay ring and $x_n$ is a non-zero divisor. Thus,
 we have that $T_d$ and the quotient $T_d/(x_n)$ have the same Betti numbers; moreover, we have that 
 $$T_d/(x_n)=\CC[x_0,\ldots,x_{n-1}]/(x_0^2,\ldots,x_{n-1}^2)^d,$$
 and the resolution of those quotients are very well known. 
 The quotient ring $\CC[x_0,\ldots,x_n]/(x_0,\ldots,x_{n-1})^d$ has a {\it pure resolution} of type $(d,d+1,\ldots,d+n-1)$
 and its Betti numbers and multiplicity are expressed with an explicit formula, see Theorem $4.1.15$ in \cite{BH}.
 
 Thus, $T_d/(x_n)$ has a {\it pure resolution of type} $(2d,2d+2,2d+4,\ldots,2d+2(n-1)),$
 i.e. $$\ldots\longrightarrow S(-2d-4)^{\beta_{3,2d+4}}\longrightarrow S(-2d-2)^{\beta_{2,2d+2}}\longrightarrow S(-2d)^{\beta_{1,2d}}\longrightarrow 0,$$
 where $S$ is the graded polynomial ring $\CC[x_0,\ldots,x_{n-1}]$ and $S(-i)$ is its $i^{th}$-shifting, i.e. $[S(-i)]_j:=S_{j-i}$.
 Moreover, the Betti numbers and the multiplicity of the quotient are given by the following formulas,
 
 \begin{align*}
  \beta_i:=\beta_{i,2d+2(i-1)}&=(-1)^{i+1}\prod_{j\neq i}\frac{d+j-1}{j-i}= \\
  &=\cancel{(-1)^{i+1}}\frac{d(d+1)\cdots (d+i-2)}{\cancel{(-1)^{i-1}}(i-1)!}\cdot\frac{(d+i)\cdots(d+n-1)}{(n-i)!}= \\
  &={{d+i-2}\choose{i-1}}{{d+n-1}\choose{n-i}}; \\ 
  e(T_d)&=\frac{1}{n!}\prod_{i=1}^n(2d+2(i-1))=2^n{{d+n-1}\choose{n}}. \\
 \end{align*}
 
 From the Betti numbers, we can easily get the Hilbert series of $T_d=S/J_d$,
 $$\HS(T_d;t)=\frac{1+\sum_{i=1}^n(-1)^i\beta_it^{2d+2(i-1)}}{(1-t)^{n+1}}.$$
 \end{proof}
 
 \begin{corollary} \label{HST}
  Let $T_d=\CC[x_0,\ldots,x_n]/J_d$, where $J_d=(x_0^2-x_n^2,\ldots,x_{n-1}^2-x_n^2)^d$, then 
   \begin{equation*}
  \HF(T_{d},m)=
   \begin{cases}
    {{n+m}\choose{n}} & \text{for }m\leq 2d-1 \\
    {{n+2d}\choose{n}}-{{d+n-1}\choose{n-1}} & \text{for }m=2d \\
    {{n+2d+1}\choose{n}}-(n+1){{d+n-1}\choose{n-1}} & \text{for }m=2d+1 \\
    2^n{{n+d-1}\choose{n}} & \text{for }m \gg 0 \\
   \end{cases}
 \end{equation*}
 \end{corollary}
 
 \begin{proof}
  The values of the Hilbert function for $m\leq 2d+1$ follow directly by extending the Hilbert series computed in Lemma \ref{HS&Betti},
  recalling that $\frac{1}{(1-t)^{n+1}}=\sum_{i\geq 0}{{n+i}\choose{n}}t^i$. Moreover, since $T_d$ is a $1$-dimensional CM ring, we have
  that its Hilbert function is eventually constant and equal to the multiplicity.
 \end{proof}
 
 Now, we are able to complete our study of the ideal of fat points with support on the $(\pm1)$-points in $\PP^n$ and prove the 
 Theorem \ref{pmPointsThm} for those points. 
 
 \begin{teo}\label{HS&BettiTeo}
  Let $I^{(d)}$ be the ideal associated to the scheme of fat points of multiplicity $d$ and support on the
  $2^n$ points $[1:\pm1:\ldots:\pm1]\in\PP^n$. The generators are given by the monomials of degree $d$ made with the $n$ quadrics
  $\Q_i=x_i^2-x_n^2$, for all $i=0,\ldots,n-1$, and the Hilbert series is 
  $$\HS\left(S/I^{(d)};t\right)=\frac{1+\sum_{i=1}^n(-1)^i\beta_it^{2d+2(i-1)}}{(1-t)^{n+1}},$$
  where the Betti numbers are given by 
  $$\beta_i:=\beta_{i,2d+2(i-1)}={{d+i-2}\choose{i-1}}{{d+n-1}\choose{n-i}},~~ \text{for }i=1,\ldots,n.$$
 \end{teo}
 
 \begin{proof}
  Let's write $I^{(d)}=J_d+J$ where $J_d=(\Q_0,\ldots,\Q_{n-1})^d$. From Lemma \ref{HS&Betti}, it is enough to show that $J=0$.
  We consider the quotient $T_d=S/(I^{(d)}+(x_n))=\CC[x_0,\ldots,x_{n-1}]/((x_0^2,\ldots,x_n)^d+\bar{J})$ and the exact sequence
  $$0 \longrightarrow \mathrm{Ann}(x_n) \longrightarrow S/I^{(d)} \overset{\cdot x_n}{\longrightarrow} S/I^{(d)} \longrightarrow T_d \longrightarrow 0.$$
  Consequently, we get $$\HS(T_d;t)=(1-t)\HS(S/I^{(t)};t)+\HS(\mathrm{Ann}(x_n);t).$$
  Since $S/I^{(d)}$ is $1$-dimensional ring, we have that $\HS(S/I^{(t)};t)=\frac{h(t)}{(1-t)}$ and the multiplicity is given by 
  $e(S/I^{(d))}=h(1)$. Thus, the multiplicity of $T_d$ is given by 
  \begin{equation}\label{eq1}
   e(T_d)=h(1)+\HS(\mathrm{Ann}(x_n);1)\geq e(S/I^{(d)})=2^n{{d+n-1}\choose{n}};
  \end{equation}
  moreover, the equality holds if and only if $x_n$
  is a non-zerodivisor of $T_d$. On the other hand, we have that $T_d=\CC[x_0,\ldots,x_{n-1}]/(x_0^2,\ldots,x_{n-1})^d+\bar{J}$ and 
  consequently, by Lemma \ref{HS&Betti}, we have 
  \begin{equation}\label{eq2}
   e(T_d)\leq e\left(\CC[x_0,\ldots,x_{n-1}]/(x_0^2,\ldots,x_{n-1})^d\right)=2^n{{d+n-1}\choose{n}};
  \end{equation}
  where equality holds if and only if $\bar{J}=0$. From \eqref{eq1} and \eqref{eq2}, we can conclude that 
  \begin{itemize}
   \item $x_n$ is a non-zerodivisor for $T_d=S/I^{(d)}$;
   \item $\bar{J}=0$.
  \end{itemize}

  Now, let's assume $J\neq 0$ and take a non-zero element $f\in J$ of minimal degree in $J$. Then, since $\bar{J}=0$, we get that 
  $f=x_n\cdot g$, for some $g$, thus we have $x_n\cdot g=0$ in $T_d$. This contradicts that $x_n$ is a non-zerodivisor in $T_d$,
  since $g\notin J$ because of minimality of $f$ in $J$ and $g\notin J_d$
  because $f$ is not.
 \end{proof}

 \begin{remark}
  In the last decades, the study of the behavior between {\it symbolic} and {\it regular powers} of homogeneous ideals involved many mathematicians
  and different areas. By definition, we always have the
  inclusion $I^m\subset I^{(m)}$, but the equality is not always true. Consequently, people started to study {\it containment
  problems}, as in \cite{ELS01} and \cite{HH02}. In \cite{BH10}, the author showed that for any $c<n$, there exists an ideal of points
  in $\PP^n$ such that $I^{(m)}\cancel\subset I^{r}$ for some $m>cr$. In \cite{BCH14}, there is a list of open conjectures regarding
  this containment problems. The authors showed also that all the conjectures hold in 
  case of equality between symbolic and regular powers $I^{(m)}=I^m$ for any $m$. 
  
  Our ideals of points in $\PP^n$ satisfy always the equality between symbolic and regular powers; consequently, they satisfy all the 
  conjectures listed in \cite{BCH14}.
  
  \smallskip{}  
  Even from the point of view of {\it Gr\"obner basis}, our result is very useful. Fixed an ordering on the 
  variables, a Gr\"obner basis for the ideal $I$ is simply a set of generators
  such that their {\it initial terms} generate the {\it initial ideal} ${\rm in}(I)$; see e.g. \cite{CLO}. 
  We recollect such properties in the following.
  
  \begin{corollary}\label{Properties}
   Let $I^{(d)}$ be the ideal of fat points of multiplicity $d$ supported on the $(\pm1)$-points of $\PP^n$. Then, we 
   have the equality between $I^{(d)}=I^d$. Moreover, for any ordering such that $x_n>x_i$ for all
   $i=0,\ldots,n-1$, the set of generators given in Theorem \ref{HS&BettiTeo} is actually a Gr\"obner basis for $I^{(d)}$. 
  \end{corollary}
  \begin{proof}
   It follows directly from Theorem \ref{HS&BettiTeo}, since we have that 
   
   \centerline{$I=I^{(1)}=(x_0^2-x_k^2,\ldots,x_{n-1}^2-x_n^2)$.}
   
   Moreover, considering the $\cG_i$'s, i.e. the set of generators obtained by taking all the possible monomial of degree 
   $d$ in the quadrics $x_i-x_n$, for all $i=0,\ldots,n-1$, we have that their leading terms generate the initial ideal, i.e.
   they are a Gr\"obner basis. Indeed, we clearly have the inclusion 
   $$\left(\textrm{in}(\cG_i)\right)\subset \textrm{in}(I);$$ but, we also have that the left hand side is exactly 
   $\left(\textrm{in}(\cG_i)\right)=(x_0^2,\ldots,x_{n-1}^2)^d$, which has the same Hilbert function of $I$, as we have seen in the
   proof of Theorem \ref{HS&BettiTeo}, and consequently the same Hilbert function of ${\rm in}(I)$. Hence, the equality holds.
  \end{proof}
 \end{remark}

 \subsection{The $k>2$ case.} Let $\xi$ be a $k^{th}$-root of unity and consider the ideal $I^{(d)}_k$
 corresponding to the scheme of fat points of multiplicity $d$ and support on the $k^n$ $\xi$-points of type 
 $[1:\xi^{g_1}:\ldots:\xi^{g_n}]\in\PP^n$ with $0\leq g_i\leq k-1$, for all $i=1,\ldots,n$.
 
 \smallskip{}
 In Section \ref{HS2}, we have considered the power ideals $I_{n,k,d}$ related to such points where the powers where only multiples of $(k-1)$.
 Thus, we cannot hope to get the Hilbert series of our scheme of fat points directly from our previous results on the Hilbert series
 of $R_{n,k,d}=S/I_{n,k,d}$. However, we can
 easily observe the following, $$\HF\left(I^{(d)},kd-1\right)=\HF\left(R_{n,k,d},kd-1\right);$$ from Remark \ref{kd-1 Rem}, we get that, assuming true the Hilbert
 function of $R_{k,d}$ conjectured, the ideal $I^{(d)}_k$ should be generated {\it at least} in degree $kd$. Thus, inspired by the $k=2$ case,
 we can actually claim that $I^{(d)}_k$ is nonzero in degree $kd$. Indeed, we have that,
 for any variable $x_0,\ldots,x_{n-1}$, we can consider the $k$ hyperplanes
 $$H_i^0=\{x_i-x_n=0\},~~H_i^1=\{x_i-\xi x_n=0\},\ldots,~~H_i^{k-1}=\{x_i-\xi^{k-1}x_n=0\};$$
 such hyperplanes divide the $k^n$ points in $k$ distinct groups of $k^{n-1}$ points; thus, their products give a set of degree $k$ forms
 which vanish with multiplicity $1$ at each point, i.e. $$\Q_i = H_i^0\cdot H_i^1\cdots H_i^{k-1} = x_i^k-x_n^k, \text{  for all }i=0,\ldots,n-1.$$
 
 Consequently, we get $$J_{k,d}=(\Q_0,\Q_1,\ldots,\Q_{n-1})^d\subset I^{(d)}_k.$$
 
 Now, by using the same ideas as for the $k=2$ case, we can get the analogous of Lemma \ref{HS&Betti} and Theorem \ref{HS&BettiTeo} 
 for all $k\geq 2$ and consequently we get the following general result.
 
 \begin{teo}\label{HS&BettiTeok}
  Let $I_k^{(d)}$ be the ideal associated to the scheme of fat points of multiplicity $d$ and support on the
  $k^n$ $\xi$-points $[1:\xi^{g_1}:\ldots:\xi^{g_n}]\in\PP^n$ for $0\leq g_i\leq k-1$. The generators are given by the monomials of 
  degree $d$ made with the $n$ forms of degree $k$
  $\Q_i=x_i^k-x_n^k$, for all $i=0,\ldots,n-1$ and the Hilbert series is 
   $$\HS\left(S/I_k^{(d)};t\right)=\frac{1+\sum_{i=1}^n(-1)^i\beta_it^{kd+k(i-1)}}{(1-t)^{n+1}},$$
   where the Betti numbers are given by 
  $$\beta_i:=\beta_{i,kd+k(i-1)}={{d+i-2}\choose{i-1}}{{d+n-1}\choose{n-i}},~~ \text{for }i=1,\ldots,n.$$
 \end{teo}

 \begin{remark}
  Moreover, similarly as for Corollary \ref{Properties}, we have that 
  \begin{itemize}
   \item $I_{k}^{(d)}=I_k^d$;
   \item the set of generators given in the theorem above, is a Gr\"obner basis.
  \end{itemize}
 \end{remark}

 \begin{remark}\label{check}
  Since we have explicitly computed the Hilbert series of $\xi$-points in $\PP^n$, by using again Macaulay duality, we can go back
  to look at the Hilbert series of the power ideals $I_{n,k,d}$. In particular, we can check that our Conjecture \ref{Conj2} holds
  in a lot of cases.
 
  Let $R_{n,k,d}$ be the quotient ring $S/I_{n,k,d}$ where $I_{n,k,d}$ is the power ideal generated by all the $(x_0+\xi^{g_1}x_1+\ldots+\xi^{g_n}x_n)^{(k-1)d}$
  with $0\leq g_i\leq k-1$ for all $i=1,\ldots,n$; and let $I^{(d)}_k$ be the ideal associated to the scheme of fat points of multiplicity
  $d$ and support on the $\xi$-points of $\PP^n$. 
  
  Now, we have seen in Section \ref{HS2} that, since $I_{n,k,d}$ is generated in degree $(k-1)d$ and generate the whole space in degree
  $kd-1$, the Hilbert function of $R_{n,k,d}$ has to be computed only in the degrees $i=(k-1)d+j$, with $j=0,\ldots,d-2$. 
  In that degrees, by Macaulay duality, we get
  \begin{equation*}
   \HF(R_{n,k,d};i)=\HF\left(I_k^{(j+1)};i\right).
  \end{equation*}
  
  From Theorem \ref{HS&BettiTeok}, we can explicitly compute such Hilbert function, i.e. for all $j=0,\ldots,d-2$,
  \begin{align}
   \HF&(R_{n,k,d};i)= \label{comp} \\
   &=\sum_{\substack{s \in \NN \\ s\leq \frac{k-1}{k}(d-j)}} (-1)^{s+1}{{n+(k-1)(d-j)-ks}\choose{n}}{{j+s-1}\choose{s-1}}{{j+n}\choose{n-s}}. \nonumber
  \end{align}
  
  In Section \ref{HS2}, we conjectured an extension of our formula for the Hilbert series of the quotient $R_{n,k,d}$ based on a 
  $\ZZ_k^{n+1}$-grading on the polynomial ring. We may recall the formula conjectured: for all $j=0,\ldots,d-2$,
  \begin{align}\label{comp2}
   \HF(R_{n,k,d};i)=\sum_{\substack{h<(k-1)(d-j) \\ i-h\in k\NN}} N_h\left({{n+\frac{i-h}{k}}\choose{n}}-{{n+j}\choose{n}}\right);
  \end{align}
  where $N_h$ is simply the number of vectors $\bh\in\ZZ_{k}^{n+1}$ of weight $\wt(\bh)=h$, see Remark \ref{AlgConj}. 
  In order to show that formula \eqref{comp2} is right and then to prove Conjecture \ref{Conj2}, we should show that the right hand side of such formula 
  is equal to the right hand side of formula \eqref{comp}. 
\end{remark}

\begin{prop}\label{equality}
 Assuming $n=1$, i.e. in the two variables case, the formulas \eqref{comp} and \eqref{comp2} are equal and Conjecture \ref{Conj2} is true.
\end{prop}\label{n1 equality}
\begin{proof}
 For any $k$ and $d$, the unique non-zero addend is the one for $s=1$; thus, 
\begin{equation*}
 \eqref{comp}=1+(k-1)(d-j)-k.
\end{equation*}

\noindent Now, we look at formula \eqref{comp2}. First of all we may observe that, for $n=1$, the number of vectors in $\ZZ_k^2$ with 
fixed weight $h$ can be computed very easily, indeed
\begin{equation*}
 N_h= \begin{cases}
 h+1 & \text{ for }0\leq h\leq k-1; \\
 2k-(h+1) & \text{ for }k\leq h \leq 2(k-1). \\
\end{cases}
\end{equation*} 

\noindent Thus, any $i=(k-1)d+j$ can be written as $ck+r$ for some positive integers $c,r$ with $0\leq r\leq k-1$ and then, we get 
{\small \begin{align*}
 \eqref{comp2}&=N_r(1+c-(j+1))+N_{r+k}(1+(c-1)-(j+1))= \\
 &=(r+1)(1+c-(j+1))+(k-r-1)(1+(c-1)-(j+1))= \\
 &=\cancel{(r+1)c}+r+1-\cancel{(r+1)(j+1)}+kc-\cancel{(r+1)c}-kj-k+\cancel{(r+1)(c+1)};
\end{align*}}
\noindent moreover, recalling that $i=ck+r=(k-1)d+j$, we finally get $$\eqref{comp2}=1+(k-1)d+j-kj-k=1+(k-1)(d-j)-k.$$
\end{proof}
  
  \begin{remark}
   With similar, but longer and more intricate arguments as for Proposition \ref{n1 equality}, we have been able to check also the case $n+1=3$.
  Unfortunately, we have been not able to prove that the two expressions given in \eqref{comp} and \eqref{comp2} give the same answer 
  for any possible parameters $(k,n,d)$. With the support of a computer, by 
  implementing with the CoCoA5 language such formulas, we have been able to check all the cases $n,k\leq 20,~d\leq 150.$
  
  \smallskip{}
  Here is the implementation of the formula \eqref{comp} by using CoCoA5 language, for the formula \eqref{comp2}, we have used the algorithm 
  described in Section \ref{algorithm}.
  
 \begin{verbatim}
-- 1) Input of the parameters K, N, D;
 K := ;
 N := ;
 D := ;
 DD := (K-1)*D;

-- HF will be the vector containing the relevant part of 
-- the Hilbert function, i.e. from (K-1)D to KD-2;
 HF := [];

-- 2) Compute the Hilbert function;
 Foreach J In 0..(D-2) Do
  B := 0;
  KK := (K-1)*(D-J)/K;
  Foreach S In 1..N Do
   If S <= KK Then
    B :=
  B+(-1)^(S+1)*Bin(N+(K-1)*(D-J)-K*S,N)*Bin(J+S-1,S-1)*Bin(J+N,N-S);
   EndIf;
  EndForeach;

  Append(Ref HF , B );
 EndForeach;

-- 3) Print the Hilbert function;
 HF;
 \end{verbatim}
 \end{remark}

\end{document}